\documentclass[12pt,a4paper,oneside]{amsart}
\usepackage[utf8]{inputenc}

\usepackage{amsmath,amscd,hyperref}
\usepackage{amsfonts}
\usepackage{amssymb}
\usepackage{amsthm}
\usepackage{cite}
\usepackage[margin=1in]{geometry}
\usepackage{mathtools}
\DeclarePairedDelimiter{\abs}{\lvert}{\rvert}
\usepackage{footnote}
\usepackage{colonequals}

\DeclareMathOperator{\im}{\mathrm{Im}}
\numberwithin{equation}{section}

\theoremstyle{plain} \newtheorem{thm}{Theorem}[section]
\newtheorem{lemma}[thm]{Lemma}
\newtheorem{prop}[thm]{Proposition}
\newtheorem{cor}[thm]{Corollary}

\theoremstyle{definition}
\newtheorem{remark}[thm]{Remark}

\newtheorem{definition}[thm]{Definition}
\newtheorem{set-up}[thm]{Set-Up}

\newtheorem{problem}[thm]{Problem}

\newtheorem*{acknowledgements*}{Acknowledgements}

\DeclareMathOperator{\rank}{rank}

\DeclareMathOperator{\Lie}{Lie}
\newcommand{\C}{\mathbb{C}}
\renewcommand{\H}{\mathcal{H}}
\newcommand{\V}{\mathcal{V}}
\newcommand{\Z}{\mathbb{Z}}
\newcommand{\R}{\mathbb{R}}
\newcommand{\G}{\mathbf{G}}
\newcommand{\g}{\mathfrak{g}}
\newcommand{\M}{\mathbf{M}}
\newcommand{\p}{\mathfrak{p}}
\newcommand{\Id}{\mathrm{Id}}
\newcommand{\norm}[1]{\left\lVert#1\right\rVert}

\newcommand{\Hol}{\mathrm{Hol}}
\newcounter{tmp}

\begin{document}

\title{$L^2$-minimal extensions over Hermitian symmetric domains}
\author{Ruijie Yang}
\address{Department of Mathematics, Stony Brook University, Stony Brook, NY 11794}
\email{ruijie.yang@stonybrook.edu}

\begin{abstract}
In this paper, we study the $L^2$-minimal extension problem for polarized variations of Hodge structures over Hermitian symmetric domains.  We are able to explicitly find the $L^2$-minimal extensions using a group-theoretic construction. In particular, this gives a construction without using $L^2$-estimates as in the Ohsawa-Takegoshi type extension theorems. The key ingredient is the Harish-Chandra embedding of Hermitian symmetric domains.

The construction of holomorphic sections might be of independent interest since it gives a concrete description in the setting of Hermitian VHS.
\end{abstract}

\maketitle

\thispagestyle{empty}

\section*{Introduction}
The purpose of this paper is to study Ohsawa-Takegoshi type extension theorems in the context of Hodge theory. Specifically, we construct $L^2$-minimal extensions for the smallest pieces of polarized variations of Hodge structures over Hermitian symmetric domains using a group-theoretic construction. This is a continuation of \cite{SY20}, where the authors prove some metric positivity results for the smallest pieces of Hodge modules over any complex manifold. In this paper, we focus on Hermitian VHS and we are able to prove a Ohsawa-Takegoshi type theorem without using $L^2$-estimates.

Turning to a more detailed description, let $D$ be a Hermitian symmetric domain and let $\V$ be a Hermitian VHS on $D$ (  Definition \ref{def:Hermitian VHS}). Let $E$ be the smallest nonzero piece of $\V$. Our main result says that $L^2$-minimal extensions of $E$ is governed by group actions on $D$.

\begingroup
\setcounter{tmp}{\value{thm}}
\setcounter{thm}{0} 
\renewcommand\thethm{\Alph{thm}}
\begin{thm}\label{thm:main} 
Let $o\in D$ be a reference point and let $v_o \in E_o$ be a vector. 
\begin{enumerate}

\item There is a holomorphic section $\sigma \in H^0(D,E)$ constructed group-theoretically such that $\sigma(o)=v_o$.

\item The $L^2$-norm
\[ \norm{\sigma}^2 \colonequals \int_{D} h_E(\sigma,\sigma) d\mu \]
is minimal among all possible holomorphic extensions of $v_o$. Here $h_E$ is the Hodge metric on $E$ and $d\mu$ is the standard Lebesgue measure on $D$ induced via the Harish-Chandra embedding.

\end{enumerate}

\end{thm}
\endgroup

Our result applies to variations of Hodge structures coming from geometry. Specifically, let $D$ be a Hermitian symmetric domain and let $\pi:X\to D$ be a smooth projective family with $n$-dimensional fibers such that $R^n\pi_{\ast}(\underline{\C})$ is a Hermitian VHS (see \cite[\S 2.4]{FL}). Note that $\pi_{\ast}(\omega_{X/D})$ is the smallest piece of $R^n\pi_{\ast}(\underline{\C})$ if it is nonzero.

\begingroup
\setcounter{tmp}{\value{thm}}
\setcounter{thm}{1} 
\renewcommand\thethm{\Alph{thm}}
\begin{cor}\label{thm:geometric case}
Let $o\in D$ be a reference point and let $v_o$ be a holomorphic $(n,0)$-form on $X_o\colonequals \pi^{-1}(o)$. 
\begin{enumerate}

\item There is a holomorphic relative $(n,0)$-form $\sigma \in H^0(D,\pi_{\ast}(\omega_{X/D}))$ constructed as in Theorem \ref{thm:main} such that $\sigma(o)=v_o$.

\item Set $c_{n}=2^{-n}(-1)^{n^2/2}$. Then the $L^2$-norm 
\[ \norm{\sigma}^2 \colonequals \int_{D} \left(c_{n} \int_{X_t} \sigma(t)\wedge \overline{\sigma(t)}\right) d\mu \]
is minimal among all possible holomorphic extensions of $v_o$, where $d\mu$ is the standard Lebesgue measure on $D$ induced via the Harish-Chandra embedding.
\end{enumerate}
\end{cor}
\endgroup

Let us illustrate our construction in an example. Let $\pi:X \to \H$ be the universal family of elliptic curves over the upper half plane. Let $i\in \H$ be a reference point. Let $X_i=\C/(\mathbb{Z} + i \mathbb{Z})$ be the standard elliptic curve and let $v_i=dz$ be the holomorphic $1$-form on $X_i$. Then our construction (see \S \ref{sec:elliptic curve example}) says that the $L^2$-minimal extension of $v_i$ is a holomorphic $1$-form $\sigma$ where
\[ \sigma(\tau)=\frac{2i}{\tau+i}(dz)_\tau. \]
Here $\tau \in \H$ and $(dz)_\tau$ is the holomorphic $1$-form on $X_\tau=\C/(\mathbb{Z} + \tau \mathbb{Z})$ descended from $\C$. One can do similar calculations for universal families of principally polarized abelian varieties over the Siegel upper half space $\H_g$.


The proof of Theorem \ref{thm:main} utilizes symmetries of $D$. Given $v_o\in E_o$, the holomorphic section $\sigma$ is constructed using $\mathrm{Exp}(\g^{-1,1}_\C)$ on the image of $D$ under the Harish-Chandra embedding (Proposition \ref{prop:equivariant sections}), where $\g^{-1,1}_\C$ is the $(-1,1)$ part of the Hodge decomposition of $\g_\C =\Lie \G_\C$. By the standard variational method, the $L^2$-minimality of $\sigma$ follows from some vanishing of integrals over $D$, which is deduced by studying the $S^1$-actions on $D$.

The construction of holomorphic sections of the smallest piece of Hermitian VHS seems to be new and it might be useful for other situations if one need a more concrete description of holomorphic sections.

Let us discuss the connection between our results and Ohsawa-Takegoshi $L^2$-extension type theorems, which was the starting point of our project. Let $B$ be a complex unit ball and let $(E,h_E)$ be a holomorphic vector bundle on $B$ with a Nakano semipositive smooth Hermitian metric $h_E$. The \textit{sharp} Ohsawa-Takegoshi theorem recently proved by Blocki and Guan-Zhou \cite{Blocki,GZ} says that if $o\in B$ is a reference point, for any $v_o \in E_o$, there is a holomorphic section $\psi \in H^0(B,E)$ such that $\psi(o)=v_o$ and $\norm{\psi}^2 \leq \mu(B)\cdot \norm{v_o}_{h_E}^2$.
Here $d\mu$ is the standard Lebesgue measure. This gives an upper bound for the $L^2$-minimal extension $\sigma$ of $v_o$, i.e.
\[ \frac{\norm{\sigma}^2}{\norm{v_o}_{h_E}^2} \leq \mu(B).\]
One can ask if more restrictions on $(E,h_E)$ can lead to sharper bounds. For example, if $\Delta \subset \C$ is the unit disk and $(E,h_E)=(\mathcal{O}_\Delta,e^{-\varphi})$ where $\varphi$ is a subharmonic function on $\Delta$, Genki \cite{Genki} proves a $L^2$-extension theorem depending on the weight function $\varphi$. Moreover, when $\varphi$ is radial, his estimates provide the $L^2$-minimal extensions. 

On the one hand, our result provides a geometric way of constructing $L^2$-minimal extensions without using $L^2$-estimates. On the other hand, our $L^2$-minimal extensions generalize Genki's $L^2$-minimal extensions to higher dimensional bases (where we allow arbitrary bounded symmetric domains). Let us explain the relation using the example of universal family of elliptic curves. Consider the biholomorphism between $\Delta$ and $\H$ where $t$ is sent to $i(1+t)/(1-t)=\tau$, then the $L^2$-minimal section we mentioned earlier becomes $\sigma(t)=(1-t)(dz)_t$ such that
\[ \norm{\sigma(t)}^2= 1-\abs{t}^2.\]
Since the extension problem for $(1,0)$-forms on $X_0$ is the same as extension problem for $(\mathcal{O}_{\Delta}, e^{-\varphi})$ where 
\[ \varphi= -\log (\norm{\sigma(t)}^2) = -\log(1- \abs{t}^2). \]
In particular, the weight function of this $L^2$-extension problem is a radial function on $\Delta$. More generally, we prove that polarization functions of holomorphic sections of Hermitian VHS are always $S^1$-invariant (Corollary \ref{cor:polarization is $S^1$-invariant}). 

Lastly, it is interesting to investigate the minimal constant $\norm{\sigma}^2/{\norm{v_o}_{h_E}^2}$ and see how does it depend on the geometry of the family. In \S \ref{sec:elliptic curve example}, we calculate the constant for universal families of elliptic curves over the upper half plane, which is $\pi/2$. Note that it is strictly smaller than $\pi$ as predicted by the sharp Ohsawa-Takegoshi Theorem.

In $\S \ref{sec:L^2}$, we state the $L^2$-minimal extension problem and formulate a variational criterion. In $\S \ref{sec:HSD}$, we review the construction of Hermitian VHS. In $\S \ref{sec:construction of sections}$, we describe the construction of holomorphic sections of Hermitian VHS and establish $S^1$-invariance property of polarization functions. In $\S \ref{sec:proof}$, we prove the main theorem. In $\S \ref{sec:elliptic curve example}$, we discuss the example of universal families of elliptic curves.

\begin{acknowledgements*}
We would like to thank Christian Schnell for his encouragement and valuable discussions. We are grateful to Nathan Chen, Robert Lazarsfeld, Radu Laza, Colleen Robles, Bowen Zhang and Zheng Zhang for useful conversations.
\end{acknowledgements*}

\section{$L^2$-minimal extension problem}\label{sec:L^2}

Let $D$ be a Stein manifold, i.e. $D$ is a complex submanifold of some $\C^N$. Let $o \in D$ be a reference point. We will describe the $L^2$-minimal extension problem in various settings.

\subsection{General set-up}
Let $(E,h_E)$ be a holomorphic vector bundle on $D$ with a smooth Hermitian metric $h_E$. The induced Hermitian inner product on the $H^0(D,E)$ is defined by 
\[ (\sigma_1,\sigma_2)_{L^2} \colonequals \int_D h_E(\sigma_1,\sigma_2) d\mu,\]
for $\sigma_1,\sigma_2 \in H^0(D,E)$ and $d\mu$ is the standard Lebesgue measure on $D$ induced from $\C^N$. The $L^2$-norm of $\sigma$ is defined by
\[ \norm{\sigma}^2\colonequals (\sigma,\sigma)_{L^2}.\]
The $L^2$-minimal extension problem is stated as follows:

\begin{problem}\label{problem:minimal extension}
For any $v_o \in E_o$, find a holomorphic section $\sigma \in H^0(D,E)$ such that 
\begin{itemize}
\item $\sigma(o)=v_o$.
\item $\norm{\sigma}^2$ is minimal among all holomorphic extensions of $v_o$.
\end{itemize}
\end{problem}

\begin{remark} \label{remark:Nakano}
If $h_E$ is Nakano semi-positive (see \cite[Ch. VII]{Demailly}), the Ohsawa-Takegoshi theorem \cite{GZ} guarantees that there is a holomorphic extension $\psi$ such that $\norm{\psi}^2$ is bounded above. In particular, the $L^2$-minimal extension of $v_o$ exists.
\end{remark}

There is a variational criterion for $L^2$-minimal extensions. 

\begin{lemma}\label{lem:minimal criterion}
Let $D$ be a Stein manifold and let $o\in D$ be a reference point. Let $(E,h_E)$ be a holomorphic vector bundle of rank $r$ on $D$ with a smooth Hermitian metric. Let
$\{ e_1,\ldots,e_r \}$ be a holomorphic frame of $E$. Then a $L^2$ holomorphic section $\sigma \in H^0(D,E)$ with $\sigma(o)=v_o$ is the $L^2$-minimal extension of $v_o$ if and only if it satisfies
\[  (fe_j,\sigma)_{L^2}=0 \]
for each $1\leq j\leq r$ and any holomorphic function $f$ on $D$ vanishing at $o$ such that $\norm{fe_j}^2 <\infty$. 

\end{lemma}

\begin{proof}
This is a standard argument for minimizers of $L^2$-norms. Let $\psi \in H^0(D,E)$ be a section vanishing at $o$ with $\norm{\psi}^2< \infty$, then minimality of $\sigma$ implies that
\[ \norm{\sigma + \epsilon \psi}^2 \geq \norm{\sigma}^2, \forall \epsilon \in \C. \]
Thus $ (\psi,\sigma)_{L^2} =0$.
In particular, we have
\begin{eqnarray}\label{eqn:vanishing1}
 (fe_j,\sigma)_{L^2}=0
\end{eqnarray}
for each $1\leq j\leq r$ and any holomorphic function $f$ on $D$ vanishing at $o$ with $\norm{fe_j}^2 <\infty$.

Conversely, let $\sigma$ be a $L^2$ holomorphic section extending $v_o$ satisfying the equation (\ref{eqn:vanishing1}). Let $\zeta$ be any other holomorphic extension. Since $\zeta -\sigma$ vanishs at $o$, by writing it in terms of the holomorphic frame, we must have
\[ (\zeta-\sigma,\sigma)_{L^2} =0. \]
Hence
\[ \norm{\zeta}^2 = \norm{\zeta - \sigma}^2 + \norm{\sigma}^2 \geq \norm{\sigma}^2,\]
which implies that $\sigma$ is $L^2$-minimal.

\end{proof}

\subsection{Extensions for canonical forms of a family}\label{sec:norms of canonical forms}
Let $D$ be a Stein manifold. Let $o\in D$ be a reference point. Let $\pi: X \rightarrow D$ be a smooth projective family with $n$-dimensional fibers. For any $t\in D$, we denote $X_t\colonequals \pi^{-1}(t)$ to be the fiber over $t$. The direct image sheaf $E\colonequals \pi_{\ast}(\omega_{X/D})$ is a holomorphic vector bundle with a smooth Hermitian metric which is the $L^2$-metric on the fiber. For any integer $m$, we define a normalizing constant $c_m  \colonequals 2^{-m}(-1)^{m^2/2}$. The relevant $L^2$-norms are defined as follows:
\begin{enumerate}

\item For $v_t \in H^0(X_t,\omega_{X_t})=E_t$,
\[ \norm{v_t}^2 \colonequals c_n \int_{X_t} v_t\wedge \overline{v_t}. \]
\item For $\sigma \in H^0(D,\pi_{\ast}(\omega_{X/D}))$, then
\[ \norm{v}^2 \colonequals \int_{D} \norm{\sigma(t)}^2 d\mu \]
where $d\mu$ is the standard Euclidean measure.

\end{enumerate}
The constant $c_n$ comes from the difference between real and holomorphic coordinates. For any $t\in D$, locally on $X_t$ we can write $v_t=fdz_1\wedge \cdots \wedge dz_n$, then
\[ c_n v_t\wedge \overline{v_t}=\abs{f}^2(dx_1\wedge dy_1)\wedge \ldots \wedge (dx_n\wedge dy_n),\]
where $z_1=x_1+\sqrt{-1}y_1,\ldots,z_n=x_n+\sqrt{-1}y_n$ are local holomorphic coordinates on $X_t$.

The $L^2$-minimal extension Problem \ref{problem:minimal extension} can be reformulated as
\begin{problem}
For any $v_o \in H^0(X_o,\omega_{X_o})$, find $\sigma \in H^0(D,\pi_{\ast}(\omega_{X/D}))$ such that 
\begin{itemize}
\item $\sigma(0)=v_o$. 
\item $\norm{\sigma}^2$ is minimal among all holomorphic extensions of $v_o$.
\end{itemize}
\end{problem}

\subsection{Extensions for variations of Hodge structures}\label{sec:ext VHS}
Extensions of canonical forms in families can be formulated more generally for abstract variations of real Hodge structures. Let $D \subset \C^N$ be a Stein manifold and let $\V$ be a holomorphic vector bundle on $D$. 
\begin{definition}
We say $\V$ is a variation of real Hodge structures if 
\begin{itemize}
\item $\V$ has an integrable connection $\nabla$ and $\V$ has a filtration $F^{\bullet}\V$ by holomorphic subbundles such that $\nabla(F^p\V)\subset F^{p-1}\V \otimes \Omega^1_D$,
\item there exists a real vector space $V$ such that $\V\cong V_\C\otimes \mathcal{O}_D$, where $V_C=V\otimes_{\R}\C$.
\item $\V$ decompose in a $C^{\infty}$-way:
\[ \V = \oplus_{p+q=n} \V^{p,q}, \quad \V^{p,q}=F^p \cap \overline{F}^{q} \]
where $n$ is the weight of $\V$ and the complex conjugation is taken with respect to the real form $V$.
\end{itemize}
\end{definition}

Let $\V$ be a variation of real Hodge structures of weight $n$.
\begin{definition}\label{def:polarized VHS}
We say $\V$ is a polarized variation of real Hodge structures if
\begin{itemize}
\item $\V$ has a nondegenerate bilinear form coming from the flat extension of a bilinear form $S: V \times V \to \R$ and we let $S^{h}$ be the Hermitian form associated to $S$ where 
\[ S^h(v,w)=i^{-n}S(v,\overline{w}).\]
Then $(-1)^pS^{h}$ is positive definite on $\V^{p,q}$, 
\item $(\V_t,F^{\bullet}\V_t,S_t)$ is a polarized real Hodge structure for each $t \in D$.
\end{itemize}
\end{definition}
Let $\V$ be a polarized real variation of Hodge structures, it is equipped with a smooth Hermitian metric $h_{\V}$ which is defined by 
\[ h_{\V}=\oplus_p (-1)^pS^{h}|_{\V^{p,q}}. \]
Let $E$ be the smallest nonzero piece in the Hodge filtration of $\V$. The smooth Hermitian metric $h_{E}\colonequals h_{\V}|_{E}$ is Nakano semi-positive \cite[Lemma 7.18]{Schmid}. Therefore the $L^2$-minimal extensions exist for $(E,h_E)$ by Remark \ref{remark:Nakano}. In particular, it applies to the extension of canonical forms of a smooth projective family.

\begin{remark}
We can formulate the $L^2$-minimal extension problem more generally for polarized \textit{complex} variations of Hodge structures, but we will only deal with real variations in this paper.
\end{remark}

\section{Variation of Hodge structures over Hermitian symmetric domains}\label{sec:HSD}

In this section, we would like to recall backgrounds on Hermitian symmetric domains and review the construction of Hermitian VHS. The primary source is \cite{Deligne}, one can also find them in \cite{LZ,Milne04}.

\subsection{Equivalent characterizations of Hermitian symmetric domains}\label{sec:bounded symmetric domain}

A Hermitian symmetric domain has three equivalent characterizations:
\begin{enumerate}
\item Intrinsic analytic description as a noncompact Hermitian symmetric manifold;
\item Extrinsic analytic description as a bounded symmetric domain;
\item Algebraic description as a conjugacy class of $\mathbb{S} \to \M$ (see Theorem \ref{thm:conjugacy class}).
\end{enumerate}

In this paper, each of three descriptions will play some role. The first description is natural from the view point of period domains. The second description is useful for the set up of $L^2$-minimal extension problem and the construction of holomorphic sections of Hermitian VHS. The third description is used to construct Hermitian VHS.

\begin{definition}
Let $D$ be a connected complex manifold with a Hermitian metric $g$. $D$ is a Hermitian symmetric domain if 
\begin{itemize}
\item the holomorphic isometry group $\mathrm{Is}(D,g)$ acts transitively on $D$ and for any point $o \in D$, there exists an involution $s \in \mathrm{Is}(D,g)$ so that $s^2=\Id$ and $o$ is an isolated fixed point of $s$.
\item $D$ is noncompact and not of Euclidean type.
\end{itemize}
\end{definition}

\begin{definition}
A bounded symmetric domain $D$ (i.e. an open connected subset of $\C^N$) is a bounded subset so that
\begin{itemize}

\item the biholomorphism group $\Hol(D)$ acts transitively on $D$;
\item For any $o \in D$, there exists an involution $s \in \Hol(D)$ such that $s^{2}=\Id$ and $o$ is an isolated fixed point of $s$.
\end{itemize}
\end{definition}
 
By \cite[Chap. VIII]{Helgason}, any bounded symmetric domain is equipped with a complex structure induced from $\C^N$ and a Hermitian metric (the Bergman metric) so that it is a Hermitian symmetric domain. Conversely, any Hermitian symmetric domain can be realized as a bounded symmetric domain by the Harish-Chandra embedding (Theorem \ref{thm:HC embedding}). In particular, bounded symmetric domains are in one-to-one correspondence with Hermitian symmetric domains. 

Lastly, let us describe $D$ in terms of the theory of Shimura varieties. Let $o\in D$ be a reference point.  
\begin{lemma}(  \cite[\S 1]{LZ} and \cite[\S 3]{Gross})
\label{lem:associated objects}

\begin{enumerate}

\item There is a unique simply connected real algebraic group $\G$ such that $\G_\R$ acts transitively on $D$ and the stabilizer $K$ at $o$ is a maximal compact subgroup.

\item There is an unique algebraic map
\[ u: U(1) \to K\subset \G_\R \]
such that $u(-1)$ is the unique element of order $2$ and $u(z) $ acts on $T_oD$ as multiplication by $z^2$.

\end{enumerate}
\end{lemma}

Let $\mathbb{S}\colonequals \mathbb{G}_m\times U(1)/\langle -1\times -1 \rangle$ be Deligne's torus. As in \cite[\S 3]{Gross}, $u$ induces an algebraic homomorphism $h: \mathbb{S} \to \M=\mathbb{G}_m\times \G/\langle -1\times u(-1) \rangle$.
One can recover $D$ from $\M$ and $h$ as follows. 
\begin{thm}(  \cite[Theorem 1.21]{LZ} and \cite[Prop. 4.9]{Milne04})\label{thm:conjugacy class}
$D$ is biholomorphic to the $\M_\R^{+}$-conjugacy class of $h$.
\end{thm}

\begin{remark}
Notice that $\G$ can be embedded into $\M$ by $g \mapsto [1,g]$ such that $\M^{\mathrm{der}}=\G$.
\end{remark}

From now on, we will use the following 

\begin{set-up}\label{setup}
Let $D$ be a Hermitian symmetric domain. Let $\G$ be the simply connected real algebraic group associated to $D$ and let $h: \mathbb{S} \to \M$ be the homomorphism associated to a given point $o \in D$.  We identify $D$ with the $\M_\R^{+}$-conjugacy class of $h$, where $h\in D$ is a reference point.
\end{set-up}

\subsection{Hermitian VHS}\label{sec:VHS from representation}
Suppose we are in the Set-up \ref{setup}. Let $V$ be a real vector space and let $\rho: \M \to \mathrm{GL}(V)$ be a real algebraic representation. Deligne \cite{Deligne} showed that there is a real variation of Hodge structures on $D$ associated to $\rho$. We would like to review the construction.

Let $h':\mathbb{S} \to \M$ be any point in $D$. Let $V_\C \colonequals V\otimes_\R \C$ be the complexification of $V$, where $\M_\C$ acts on $V_\C$ via $\rho_\C$. We will write $\rho_\C=\rho$ by abusing notations. Then $\rho \circ h'$ defines a real Hodge structure on $V_\C$ such that
\[ V_\C = \bigoplus_{p,q} V^{p,q}_{h'},\]
where
\[ V^{p,q}_{h'} = \{ v \in V_\C \colon \rho(h'(z))(v) = z^{-p}\bar{z}^{-q}v, \forall z\in \mathbb{S} \}.\]
In particular, we have the weight space decomposition
\begin{eqnarray}\label{eqn:weight decomposition}
V = \bigoplus_{n\in \mathbb{Z}} V_{n,h'}, \quad V_{n,h'}\otimes_{\R} \C = \bigoplus_{p+q=n} V^{p,q}_{h'},
\end{eqnarray}
i.e. $v \in V_{n,h'}$ if and only if $\rho(h'(r))(v)= r^{-n}\cdot v$ for all $r\in \mathbb{G}_{m}(\R)=\R^{\ast}$.
\begin{remark}
Here we use the convention in the theory of Shimura varieties such that $\rho(h(z)) \cdot v = z^{-p}\bar{z}^{-q}\cdot v$.
\end{remark}

Since $D$ is a Hermitian symmetric domain, it can be shown that the weight decomposition of $V$ is independent of the choice of $h' \in D$ (  \cite[1.1.13($\alpha$)]{Deligne} and \cite[Theorem 1.20]{LZ}). Therefore, for any integer $n$, the space $V_n \colonequals V_{n,h'}$ is independent of the choice of $h'$. We consider the associated trivial bundle $\V=(V_n)_\C \otimes \mathcal{O}_D$ with the connection $\nabla=1\otimes d: (V_n)_\C \otimes \mathcal{O}_D \to (V_n)_\C \otimes \Omega^1_D$. The Hodge structure induced by $h'$ gives a filtration $F^{\bullet}\V$ on $\V$ where 
\[ (F^{p}\V)_{h'}=\oplus_{a\geq p} V^{a,n-a}_{h'}.\] 

\begin{thm}(\cite[Prop. 1.1.14]{Deligne})\label{thm:representation give VHS}
With the notations above, $(\V,F^{\bullet}\V,\nabla)$ is a variation of real Hodge structures of weight $n$.
\end{thm}

\begin{definition}\label{def:Hermitian VHS}
Let $\rho:\M \to \mathrm{GL}(V)$ be a real represention and let $n$ be an integer. We will say $\V$ constructed above is a Hermitian VHS.
\end{definition}

On the one hand, $\M_\R^{+}$ acts on $D$ via conjugactions. On the other hand, $\M_\R^{+}$ acts on $V_\C$ via the representation. These two actions are compatible in the following sense:

\begin{lemma}\label{lem:equivariance of fiber}
With the notations above. Let $g\in \M_\R^{+}$ be an element. Then for any $(p,q)$ we have
\[ V^{p,q}_{g\cdot h} = \rho(g) \cdot V^{p,q}_{h}. \]
\end{lemma}

\begin{proof}
Denote $h'$ to be the homomorphism $g\cdot h$, which is defined by $h'(z)=g \cdot h(z) \cdot g^{-1}$ for any $z\in \mathbb{S}$. Then
\begin{align*}
V^{p,q}_{g\cdot h} =V^{p,q}_{h'}&= \{ v\in V_\C \colon \rho(h'(z))(v)=z^{-p}\bar{z}^{-q}  v, \quad \forall z \in \mathbb{S} \} \\
                    &=\{  v\in V_\C \colon \rho(g\cdot h(z)\cdot g^{-1})(v)=z^{-p} \bar{z}^{-q} v, \quad \forall z \in \mathbb{S} \} \\
                    & =\{  v\in V_\C \colon \rho(h(z))(\rho(g)^{-1}(v))=z^{-p} \bar{z}^{-q} (\rho(g)^{-1}(v)), \quad \forall z \in \mathbb{S} \} \\
                   &= \rho(g)\cdot V^{p,q}_{h}. \qedhere
\end{align*} \end{proof}

\section{Holomorphic sections of Hermitian VHS}\label{sec:construction of sections}
In this section, we will use the Harish-Chandra embedding to give a group-theoretic construction of holomorphic sections of Hermitian VHS. We will also deduce a $S^1$-invariant property for polarization functions of these sections.

\subsection{Harish-Chandra embeddings}\label{sec:Harish-Chandra}
Suppose we are in the Set-up \ref{setup}. We will review the construction of Harish-Chandra embedding from the view point of Shimura varieties (\cite[\S 2.4]{Viviani} and \cite[Ch. VIII \S 7.1]{Helgason}) and prove some useful properties.
 
Consider the real Lie algebra $\g=\Lie \G_\R$. Since $\Lie \M_\R$ admits a Hodge structure of weight $0$ associated to $\mathrm{Ad}\circ h$ (\cite[Theorem 1.20]{LZ}),  the Hodge decomposition of $\Lie \M_\C$ restricting to $\g_\C$ is 
\[ \g_\C = \g^{0,0}_\C \oplus \g^{-1,1}_\C \oplus \g^{1,-1}_\C.\]
Since $\mathbb{G}_m$ acts as identity, we have
\[ \g^{0,0}_\C=\g_\C \cap (\Lie \M_\C)^{0,0}, \g^{-1,1}_\C=(\Lie \M_\C)^{-1,1},  \g^{1,-1}_\C=(\Lie \M_\C)^{1,-1}.\]
\begin{remark} 
One can also see this decomposition directly via $\mathrm{Ad}\circ u$, where $u: U(1) \to \G_\R$ is the algebraic homomorphism in Lemma \ref{lem:associated objects}.
\end{remark}
\begin{lemma}\label{lem:subalgebras}
The subspaces $\g^{0,0}_\C, \g^{-1,1}_\C$ and $\g^{1,-1}_\C$ are subalgebras of $\g_\C$. Moreover,
\[ [\g^{-1,1}_\C,\g^{-1,1}_\C]=0, \quad [\g^{1,-1}_\C,\g^{1,-1}_\C]=0, \quad [\g^{0,0}_\C,\g^{-1,1}_\C ]\subset \g^{-1,1}_\C,  \quad [\g^{0,0}_\C,\g^{1,-1}_\C]\subset \g^{1,-1}_\C.\]
In particular, $\g^{-1,1}_\C$ and $\g^{1,-1}_\C$ are abelian subalgebras of $\g_\C$.
\end{lemma}

\begin{proof}
Since $\mathrm{Ad}$ commutes with the Lie bracket on $\g_\C$, we have
\[ [\g^{p_1,q_1}_\C, \g^{p_2,q_2}_\C] \subset \g^{p_1+p_2,q_1+q_2}_\C. \qedhere\] 
\end{proof}

Now let us put the Hodge decomposition of $\g_\C$ into the setting of the Harish-Chandra embedding. Let $\g_\R=\mathfrak{k} \oplus \p$ be the Cartan decomposition induced by $\mathrm{Ad}(h(i))$, where $\mathfrak{k}$ (resp. $\p$) is the $1$ (resp. $-1$) eigenspace. In particular, $\p$ admits a complex structure via $\mathrm{Ad}(h(e^{i\pi/4}))$. Therefore we have
\[ \g_\C=\mathfrak{k}_\C \oplus \p^{i} \oplus \p^{-i},\]
where $\p^{\pm i}$ are $\pm i$ eigenspaces of $\p_\C$ with respect to $\mathrm{Ad}(h(e^{i\pi/4}))$. One can check that
\[ \mathfrak{k}_\C = \g^{0,0}_\C, \quad  \p^{i}=\g^{-1,1}_\C, \quad \p^{-i}=\g^{1,-1}_\C.\]
Let $K_\C, P_{+}$ and $P_{-}$ be the analytic subgroups of $\G_\C$ whose Lie algebras are $\g^{0,0}_\C,  \g^{-1,1}_\C, \g^{1,-1}_\C$. Let $K \subset \G_\R$ be the stabilizer of $h$ as in Lemma \ref{lem:associated objects}. We observe that $\Lie K=\mathfrak{k}$ and $K_\C$ is the complexification of $K$. Moreover, we have
\begin{lemma} (\cite[Chap. VIII, \S 7]{Helgason})\label{lem:relations}
\begin{enumerate}
\item $K_\C P_{+} K_\C^{-1} \subset P_{+}, \quad K_\C P_{-} K_\C^{-1} \subset P_{-}$.

\item The multiplication map $P_{+} \times K_\C \times P_{-} \to \G_\C$ is an injection and the image contains $\G_\R$. Also $P_{+} \cap K_\C P_{-}= \{ e\}$.

\item $P_{+}$ is an abelian Lie group and the exponential map $\mathrm{Exp}:\g^{-1,1}_\C \to P_{+}$ is a diffeomorphism. We denote its inverse by $\log: P_{+} \to \g^{-1,1}_\C$.
\end{enumerate}
\end{lemma}

The Harish-Chandra embedding is defined as follows. Let $h' \in D$ be any point. Suppose $h'=g\cdot h$ for some $g \in \G_\R$. By Lemma \ref{lem:relations}, there is a unique $\zeta(g)\in P_{+}$ so that 
\[ g\in \zeta(g)K_\C P_{-}. \]
This gives us a well-defined map
\[ \zeta: D \to P_{+}, \quad \zeta(h')\colonequals \zeta(g).\]
\begin{thm}(Harish-Chandra embedding \cite[Chap. VIII, Theorem 7.1]{Helgason})\label{thm:HC embedding} 
With the notations above, the following map 
\[ i_{\mathrm{HC}}\colon D \xrightarrow{\zeta} P_{+} \xrightarrow{\log} \g^{-1,1}_{\C}\]
is a holomorphic open embedding such that $h$ is sent to $0$ and $i_{\mathrm{HC}}(D)$ is a bounded symmetric domain in $\g^{-1,1}_{\C}$.
\end{thm}

Notice that $K\subset K_\C$ and $K_\C P_{+} K_\C^{-1} \subset P_{+}$, therefore $K$ acts on $\g^{-1,1}_\C$ via $\mathrm{Ad}$. Moreover, $i_{\mathrm{HC}}$ is $K$-invariant by the following
\begin{lemma}\label{lem:S^1 action}
Let $k\in K$ be any element and denote $k\cdot$ to be the action of $k$ on $D$. Then
\[ i_{\mathrm{HC}}\circ (k\cdot)=\mathrm{Ad}(k)\circ i_{\mathrm{HC}}. \]
\end{lemma}

\begin{proof}
Let $h'\in D$ be any point and let $g\in \G_\R$ so that $h'=g\cdot h$. Let $k\in K$ be any element. Since $g\in \zeta(g) K_{\C}P_{-}$ and $P_{-}k^{-1} \subset K_{\C}P_{-}$ by Lemma \ref{lem:relations}, we have
\[ kgk^{-1} \in k\zeta(g) K_{\C}P_{-}k^{-1}=k\zeta(g)k^{-1}K_{\C}P_{-}. \]
Because $k\zeta(g)k^{-1}\in P_{+}$, the uniqueness of $\zeta(kgk^{-1})$ implies that
\[ \zeta(kgk^{-1}) = k\zeta(g)k^{-1}.\]
Since $K$ is the stabilizer of $h$, we conclude that
\[ i_{\mathrm{HC}}(k\cdot h')=i_{\mathrm{HC}}(kgk^{-1}\cdot h)=\log (k\zeta(g)k^{-1})=\mathrm{Ad}(k)(\log \zeta(g))=\mathrm{Ad}(k)(i_{\mathrm{HC}}(h')).\]
The second to the last equality follows from the fact that the differential of $\log:P_{+} \to \g^{-1,1}_{\C}$ is an isomorphism.
\end{proof}

\subsection{Construction of holomorphic sections}\label{sec:equivariant sections}
Suppose we are in the Set-up \ref{setup}. Let $\rho: \M \to \mathrm{GL}(V)$ be a real algebraic representation. As in \S \ref{sec:VHS from representation}, for an integers $n,p$, let $V_n$ be the weight $n$ space in the decomposition (\ref{eqn:weight decomposition}) and let $F^p\V$ be the $p^{\mathrm{th}}$ Hodge filtration of $\V=(V_n)_\C \otimes \mathcal{O}_D$. The goal of this section is to give a group-theoretic construction of a holomorphic section in $H^0(D,F^p\V)$ extending any given vector in $(F^p\V)_h$.

Let $d\rho: \Lie \M \to \mathrm{End}(V)$ be the differential of $\rho$. Recall from \S \ref {sec:Harish-Chandra} that $\g^{1,-1}_\C =(\Lie \M_\C)^{1,-1}$ is the $(1,-1)$-part of the decomposition with respect to $\mathrm{Ad}\circ h$. 

\begin{lemma}\label{lem: shifting of (p,q)}
Let $X\in \g^{1,-1}_\C$. For any integer $(p,q)$ so that $p+q=n$,
\[ d\rho(X) \cdot V^{p,q}_h \subset V^{p+1,q-1}_h.\]
In particular, for $g\in \mathrm{Exp}(\g^{1,-1}_\C)$ we have
\[ \rho(g)\cdot (F^p\V)_h \subset (F^{p}\V)_h.\]
\end{lemma}

\begin{proof}
Let $g\in \mathrm{Exp}(\g^{1,-1}_\C)$ and $X\in \g^{1,-1}_\C$ so that $g=\mathrm{Exp}(X)$. We will denote $g\cdot v \colonequals \rho(g)(v)$ and $X\cdot v \colonequals d\rho(X)(v)$.

Let $z\in \mathbb{S}$ by any element. Let $v\in V^{p,q}_h$ be a vector. Since $X\in \g^{1,-1}_\C$,  we have
\[ h(z)\cdot v=z^{-p}\overline{z}^{-q}v, \quad \mathrm{Ad}(h(z))(X)=z^{-1}\overline{z}^1X. \]
Hence
\begin{align*}
h(z)\cdot (X\cdot v) &= (h(z)Xh(z)^{-1})\cdot (h(z)\cdot v) \\
                               &= \mathrm{Ad}(h(z))(X) \cdot (z^{-p}\overline{z}^{-q}v)\\
                               &=(z^{-1}\overline{z}^1X) \cdot (z^{-p}\overline{z}^{-q}v) \\
                               &= z^{-(p+1)}\overline{z}^{-(q-1)}(X\cdot v)
\end{align*}
This implies that 
\[ X\cdot v \in V^{p+1,q-1}_h.\]
In particular, $X\cdot (F^p\V)_h \subset (F^{p+1}\V)_h$. On the other hand, $V$ has only finitely $(p,q)$ components, there exists an integer $k$ so that  $d\rho(X)^{k+1}=0$.  We conclude that
\begin{align*} 
g\cdot (F^p\V)_h &= \mathrm{Exp}(X) \cdot (F^p\V)_h \\
              &= \Id \cdot (F^p\V)_h + X\cdot (F^p\V)_h + \frac{X^2}{2} \cdot (F^p\V)_h + \cdots + \frac{X^{k}}{k!} \cdot (F^p\V)_h \\
              &\subset (F^p\V)_h + (F^{p+1}\V)_h + (F^{p+2}\V)_h + \cdots (F^{p+k}\V)_h\\
              &=(F^p\V)_h \qedhere
\end{align*}\end{proof}

Now we will start the construction of holomorphic sections. Let $i_{\mathrm{HC}}:D \to \g^{-1,1}_\C$ be the Harish-Chandra embedding (Theorem \ref{thm:HC embedding}).  Let $\widetilde{\V}=(V_n)_\C \otimes \mathcal{O}_{\g^{-1,1}_\C}$ be the trivial bundle on $\g^{-1,1}_\C$. Notice that $i_{\mathrm{HC}}^{\ast}\tilde{\V} = \V$, we will first construct a holomorphic section of $\widetilde{\V}$ and its pull-back will be our desired section.
\begin{lemma}\label{lem:holomorphic sections}
Given $v \in (V_n)_\C$, we define $\tilde{\sigma}$ to be a $C^{\infty}$-section of $\widetilde{\V}$ on $\g^{-1,1}_\C$ such that for $X\in \g^{-1,1}_\C$,
\[ \tilde{\sigma}(X) \colonequals \rho(\mathrm{Exp}(X))(v). \]
Then $\tilde{\sigma}$ is a holomorphic section of $\widetilde{\V}$.
\end{lemma}
\begin{proof}
It suffices to show that $\tilde{\sigma}:\g^{-1,1}_\C \to (V_n)_\C$ is holomorphic. For any $X \in \g^{-1,1}_\C$,
\[ \rho(\mathrm{Exp}(X))=e^{d\rho(X)}.\]
A similar argument in Lemma \ref{lem: shifting of (p,q)} implies that $X\cdot V^{p,q}_h \subset V^{p-1,q+1}_h$, therefore $d\rho(X)$ is a nilpotent operator on $(V_n)_\C$ and $e^{d\rho(X)}$ is a polynomial in $d\rho(X)$. Therefore $\tilde{\sigma}$ is holomorphic.
\end{proof}

\begin{prop}\label{prop:equivariant sections}
Given $v_h \in (F^p\V)_h$, let $\tilde{\sigma}$ be the holomorphic section in $H^0(\g^{-1,1}_\C, \widetilde{\V})$ constructed in Lemma \ref{lem:holomorphic sections}. Then the pull-back section 
\[ \sigma \colonequals i_{\mathrm{HC}}^{\ast}\tilde{\sigma} \in H^0(D,\V) \]
is a holomorphic section inside $H^0(D,F^p\V)$ and it extends $v_h$.
\end{prop}

\begin{proof}
To simplify the notations, for $g\in \G_\C$ and $v\in V_\C$, we will denote $g\cdot v \colonequals \rho(g)(v)$. Let $h'$ be any point in $D$. Since $\G_\R^{+}$ acts transitively on $D$ \cite[Theorem 1.6]{LZ}, we can assume $h'=g\cdot h$ for some $g\in \G_\R^{+}$. 
By Theorem \ref{thm:HC embedding} we have 
\[ i_{\mathrm{HC}}(h')=\log \zeta(g),\]
where $\zeta(g)$ is the unique element in $P_{+}$ such that $g \in \zeta(g)K_\C P_{-}$. Therefore 
\[ \sigma(h')=\tilde{\sigma}(i_{\mathrm{HC}}(h'))= \mathrm{Exp}(\log \zeta(g))\cdot v_h=\zeta(g) \cdot v_h. \]
Since $g\in \G_\R^{+} \subset \M_\R^{+}$, Lemma \ref{lem:equivariance of fiber} implies that $(F^p\V)_{h'}= g\cdot (F^p\V)_{h}$. Then
\[ \sigma(h') \in (F^p\V)_{h'} \Longleftrightarrow \zeta(g) \cdot v_h \in g\cdot (F^p\V)_{h} \Longleftrightarrow
g^{-1}\zeta(g) \cdot (F^p\V)_h \subset (F^p\V)_h. \]
Notice that $g \in \zeta(g)K_\C P_{-}$, there exists $k_\C \in K_\C$ and $p_{-} \in P_{-}$ such that $g=\zeta(g)k_\C p_{-}$. In particular,
\[ g^{-1}\zeta(g)=p_{-}^{-1}k_\C^{-1}.\]
Since $K$ is the stabilizer of $h$, $K_\C$ also fixes $h$ and 
\[k_\C^{-1}\cdot (F^p\V)_h=(F^p\V)_h.\]
On the other hand, $p_{-} \in P_{-}=\mathrm{Exp}(\g^{1,-1}_\C)$, Lemma \ref{lem: shifting of (p,q)} implies that
\[ p_{-}^{-1}k_\C^{-1} \cdot (F^p\V)_h \subset p_{-}^{-1}(F^p\V)_h \subset (F^p\V)_h.\qedhere\]
\end{proof}

\subsection{Polarizations of Hodge bundles}\label{sec:polarizations}
Suppose we are in the Set-up \ref{setup}. Let $\rho: \M \to \mathrm{GL}(V)$ be a real algebraic representation. In particular, $V$ is a $\G_\R$-representation. In this section, we will write $gv\colonequals \rho(g)(v)$. 

Let $n$ be an integer and let $V_n$ be the weight $n$ space in the decomposition \ref{eqn:weight decomposition}. In \S \ref{sec:VHS from representation}, we see that $\V=(V_n)_\C \otimes \mathcal{O}_D$ is a variation of real Hodge structures. In this section we would like to review the construction of a canonical polarization on $\V$ such that $\V$ becomes a polarized variation of real Hodge structures. The main goal is to deduce a $S^1$-invariance of the polarization function, which should be well-known to experts, but we include the proof here for the completeness.

We would construct a bilinear form on $V$ depending on $h(-i)\in \G_\R$. Therefore it is convenient to work in a more general setting following \cite[\S 2]{Milne13}. 

\begin{definition}
Let $V$ be a $\G_\R$-representation and let $C\in \G_\R$ be an element. A $C$-polarization on $V$ is a $\G_\R$-invariant bilinear form $S: V\times V \to \R$ such that the form $S_C$ defined by $S_C(v,w)\colonequals S(v,Cw)$ is symmetric and positive definite.
\end{definition}

Let $g \mapsto \overline{g}$ denotes the complex conjugation on $\G_\C$ with respect to $\G_\R$.
\begin{definition}
An involution $\theta$ of $\G_\R$ is said to be \textit{Cartan} if the group
\[ G^{\theta} \colonequals \{ g \in \G_\C \mid g=\theta(\overline{g})\}\]
is compact.
\end{definition}

\begin{prop}(\cite[Theorem 2.1]{Milne13})\label{prop:C-polarization}
Let $C\in \G_\R$ be an element such that $C^2$ lies in the center of $\G_\R$. If $\mathrm{Inn}(C)$ is a Cartan involution, then $V$ has a $C$-polarization $S$. Moreover, let $K^{\mathrm{Inn}(C)}$ be the fixed point set of $\mathrm{Inn}(C)$. The positive definite Hermitian form 
\[ \widetilde{S}(v,w) \colonequals S(v,\overline{Cw})\]
is $K^{\mathrm{Inn}(C)}$-invariant, where
\[ \widetilde{S}(kv,kw)=\widetilde{S}(v,w),\]
for any $k \in K^{\mathrm{Inn}(C)}$.
\end{prop}

\begin{remark}
There is no $K^{\mathrm{Inn}(C)}$-invariant part in \cite[Theorem 2.1]{Milne13}, but it follows immediately from the construction in the proof.
\end{remark}

\begin{prop}\label{prop:polarization}
Suppose we are in the Set-up \ref{setup}. Let $V$ be a real algebraic representation of $\M$. Let $n$ be an integer and let $V_n$ be the weight $n$ space in the decomposition \ref{eqn:weight decomposition}. There is a bilinear form $S$ on $V_n$ such that $\V=(V_n)_\C \otimes \mathcal{O}_D$ is a polarized variation of real Hodge structures (Definition \ref{def:polarized VHS}). Moreover, $S^h$ is $K$-invariant, where
\[ S^h(v,w) \colonequals i^{-n}S(v,\overline{w}).\]
\end{prop}

\begin{proof}
Let $C=h(-i) \in \G_\R$, then $C^2=h(-1)$ lies in the center of $\G_\R$ because $\mathrm{Ad}\circ h$ has weight zero. Also we know that $\mathrm{Inn}(C)=\mathrm{Inn}(C^{-1})=\mathrm{Inn}(h(i))$ is a Cartan involution on $\G_\R$\cite[Prop 1.19]{LZ}. Proposition \ref{prop:C-polarization} implies that there is a $C$-polarization $S$ on $V$ so that $\widetilde{S}$ is a positive definite Hermitian form where $\widetilde{S}(v,w)\colonequals S(v,\overline{Cw})$.

For any $p,q$ such that $p+q=n$ we would like to show that on $ V^{p,q}_h$
\[ (-1)^pS^h=\widetilde{S}, \]
which would implies that $(-1)^pS^h$ is positive definite on $\V^{p,q}$. For $w \in V^{p,q}_h$, we have 
\[ Cw = h(-i)\cdot w=(-1)^pi^{-n}w.\]
Therefore
\[ (-1)^pS^h(v,w)=(-1)^pi^{-n}S(v,\overline{w})=S(v,\overline{Cw})=\widetilde{S}(v,w).\]
In particular, $\V$ is polarized by $S|_{V_n}$. Notice that $K$ is the fixed point set of $\mathrm{Inn}(h(i))$, the $K$-invariant part also follows from Proposition \ref{prop:C-polarization}.
\end{proof}

Lastly, we would like to show the following $S^1$-invariance property of the polarization function.
\begin{cor}\label{cor:polarization is $S^1$-invariant}
With the notations in Proposition \ref{prop:polarization}. Let $E$ be the smallest nonzero piece of $\V$ with the Hodge metric $h_E$. Let $w_1,w_2$ be any two vectors in $E_h$ and let $\sigma_1,\sigma_2$ be the two holomorphic sections in $H^0(D,E)$ extending $w_1,w_2$ as in Proposition \ref{prop:equivariant sections}. Denote 
\[ (\sigma_1,\sigma_2)_{h'}\colonequals h_E(\sigma_1(h'),\sigma_2(h')). \]
Then the polarization function is $S^1$-invariant: for $h'\in D$ and $z\in S^1$ we have
\[ (\sigma_1,\sigma_2)_{z\cdot h'} = (\sigma_1,\sigma_2)_{h'} \]
where $z\cdot h'$ is induced by $u:S^1 \to K$ as in Lemma \ref{lem:associated objects}.
\end{cor}

\begin{proof}
Let $g\in \G_\R$ be an element so that $h'=g\cdot h$. As in the construction of Proposition \ref{prop:equivariant sections}, let $\zeta(g)\in P_{+}$ be the element so that
\[\sigma_1(h')=\zeta(g) \cdot w_1, \quad \sigma_2(h')=\zeta(g) \cdot w_2.\]
Since $h(z) \in K$, by the proof of Lemma \ref{lem:S^1 action} we have
\[ \sigma_1(z\cdot h')= z\zeta(g)z^{-1}\cdot w_1, \quad \sigma_2(z\cdot h')= z\zeta(g)z^{-1}\cdot w_2.\]
Since $E$ is the lowest piece, write $E=\V^{p,q}$. Then $h_E=(-1)^pS^h$, which is $K$-invariant by Proposition \ref{prop:polarization}. For any $z\in S^1$,
\begin{align*}
(\sigma_1,\sigma_2)_{z\cdot h'}
                   &=(z\zeta(g)z^{-1}\cdot w_1,z\zeta(g)z^{-1}\cdot w_2) \\
                   &=(\zeta(g)\cdot (z^{-1}\cdot w_1), \zeta(g)\cdot (z^{-1}\cdot w_2)) \quad \quad [\text{$K$-invariance}] \\
                   &=(\zeta(g)\cdot (z^p\overline{z}^{q}w_1), \zeta(g)\cdot (z^p\overline{z}^{q}w_2)) \quad \quad [w_1,w_2\in V^{p,q}_h]\\
                   &=(\zeta(g)\cdot w_1,\zeta(g)\cdot w_2) \quad \quad [z\in S^1] \\
                   &=(\sigma_1,\sigma_2)_{h'}
\end{align*}
Therefore the polarization function is $S^1$-invariant.
\end{proof}

\begin{remark}
One can argue more generally for $w_1,w_2\in V^{r,s}_h$, the function $(-1)^rS^h(\sigma_1,\sigma_2)$ is $S^1$-invariant. But since $\sigma_1,\sigma_2$ are not sections of $\V^{r,s}$, hence
$h_\V(\sigma_1,\sigma_2) \neq (-1)^rS^h(\sigma_1,\sigma_2)$.
\end{remark}

\section{Proof of Theorem \ref{thm:main}}\label{sec:proof}

Let $D$ be a Hermitian symmetric domain. Let $o\in D$ be a reference point. Now we are in the Set-up \ref{setup}, where $\G$ is the simply connected real algebraic group associated to $D$ and $h:\mathbb{S} \to \M$ is the unique algebraic homomorphism associated to $o$. We identify $D$ with the $\M_\R^{+}$-conjugacy class of $h$ and $o$ is identified with $h$.

Let $\rho: \M \to \mathrm{GL}(V)$ be a real algebraic representation. Fix an integer $n$, by Theorem \ref{thm:representation give VHS} and Proposition \ref{prop:polarization}, let $\V=(V_n)_\C \otimes \mathcal{O}_D$ be variation of real Hodge structures of weight $n$. Let $E$ be the smallest nonzero piece in the Hodge filtration of $\V$ equipped with the Hodge metric $h_E$. The Harish-Chandra embedding realize $D$ as a bounded symmetric domain with the induced standard Lebesgue measure. Here we are in the set-up of Problem \ref{problem:minimal extension}. 

Let $v_h \in E_h$ be a vector and let $\sigma$ be the holomorphic section in $H^0(D,E)$ constructed by Proposition \ref{prop:equivariant sections} satisfying $\sigma(h)=v_h$. Note that $\sigma=i_{\mathrm{HC}}^{\ast}\widetilde{\sigma}$ and $\rho$ is algebraic, the polarization function $h_E(\sigma,\sigma)$ can be extended to a continuous function over the closure of the bounded symmetric domain. In particular, we have  $\norm{\sigma}^2<\infty$. We would like to show $\sigma$ is the $L^2$-minimal section using Lemma \ref{lem:minimal criterion}.

Assume the rank of $E$ is $r$ and let $\{ v_1,\ldots,v_r\}$ be a basis of $E_h$. We extend them to holomorphic sections $\{ e_1,\ldots,e_r \} \subset H^0(D,E)$ by Proposition \ref{prop:equivariant sections}. This is a holomorphic frame of $E$ because $\rho$ preserves the linear independence. Let $j$ be an integer with $1\leq j \leq r$ and let $f$ be a holomorphic function on $D$ vanishing at $h$ such that $\norm{fe_j}^2 < \infty$. By Lemma \ref{lem:minimal criterion}, it suffices to show that 
\begin{eqnarray}\label{eqn:vanishing}
(fe_j, \sigma)_{L^2}=\int_D h_E(fe_j, \sigma)d\mu=0
\end{eqnarray}
where $\mu$ is the standard Lebesgue measure on $D$ induced from the Harish-Chandra embedding.

The vanishing of the integral (\ref{eqn:vanishing}) follows from a symmetry argument. By Lemma \ref{lem:associated objects}, we have an embedding $u:S^1 \to K$. For $z\in S^1$, we will denote the action by $h(z)\cdot$ on different objects. 
\begin{enumerate}

\item By Corollary \ref{cor:polarization is $S^1$-invariant},
\[ h(z)\cdot h_E(e_j,\sigma) = h_E(e_j,\sigma).\]

\item Let $t_1,\ldots, t_k$ be a holomorphic coordinate on $D$ induced from $\g^{-1,1}_\C$ via the Harish-Chandra embedding. Then by Lemma \ref{lem:S^1 action}, 
\[ h(z)\cdot t_m= z^2\cdot t_m, \quad h(z)\cdot d\mu=d\mu \]
for each $1\leq m \leq k$. 
\item Consider the Taylor expansion of $f$ at $0$:
\[ f= \sum_{\alpha \geq 1} f^{\alpha},\]
where $f^{\alpha}$ is a homogeneous polynomial in $t_1,\ldots,t_k$ of degree $\alpha$. Then
\[ h(z) \cdot f^{\alpha} = z^{2\alpha}f^{\alpha}.\]
\end{enumerate}

Now suppose $W\subset D$ is a compact subset such that $W$ is $S^1$-invariant where 
\[ h(z)\cdot W=W\]
for each $z\in S^1$. Applying the change of variable for integrals induced by the action $h(z)$ on $W$, we have 
\begin{align*}
\int_W f^{\alpha}h_E(e_j,\sigma) d\mu&=\int_{h(z)\cdot W} [h(z)\cdot f^{\alpha}] [h(z)\cdot h_E(e_j,\sigma)] d\mu \\
&= \int_{W} (z^{2\alpha}f^{\alpha})h_E(e_j,\sigma) d\mu \\
&=z^{2\alpha}\int_W f^{\alpha}h_E(e_j,\sigma) d\mu.
\end{align*}
Since for each $\alpha$, we can choose some $z \in S^1$ so that $z^{2\alpha}\neq 1$, it follows that
\[ \int_W f^{\alpha}h_E(e_j,\sigma) d\mu = 0. \]
Since $W$ is compact, we can approximate $f$ with $\sum^m_{\alpha \geq 1}f^{\alpha}$ and get
\begin{eqnarray}\label{eqn: vanishing on $S^1$-invariant subsets}
 \int_W h_E(fe_j,\sigma) d\mu=0.
\end{eqnarray}
Now let $\{W_{\ell}\}$ be an exhaustion of $D$. We set
\[ W'_{\ell} = \cup_{z\in S^1} h(z)\cdot W_{\ell}. \]
Since $S^1$ is compact, $W'_{\ell}$ is a $S^1$-invariant compact subset of $D$ and vanishing (\ref{eqn: vanishing on $S^1$-invariant subsets}) implies that
\[ \int_{W'_{\ell}} h_E(fe_j,\sigma) d\mu=0.\]
Since $\{W'_{\ell}\}$ is also an exhaustion of $D$ and $(fe_j,\sigma)_{L^2}<\infty$, we have
\[ (fe_j,\sigma)_{L^2}= \int_D h_E(fe_j,\sigma) d\mu=0.\]
This implies (\ref{eqn:vanishing}) and finishes the proof of Theorem \ref{thm:main}.

\section{Example: universal family of elliptic curves}\label{sec:elliptic curve example}

In this section, we would like to discuss the $L^2$-minimal extensions for the universal family of elliptic curves over the upper half plane. We will construct the minimal extensions using the recipe described in the previous sections. Similar calculations can be done for universal families of principally polarized abelian varieties or K3 surfaces. 

Let $\H$ be the upper half plane, then $\H$ is a Hermitian symmetric domain with respect to the Poincar\'e metric. We fix $i \in \H$. 

\textbf{Step 1}. First let us work out the Set-up \ref{setup}. Since $\Hol^{+}(\H)=\mathrm{PSL}(2,\R)$. The associated real algebraic group is $\G_\R=\mathrm{SL}(2,\R)$ which acts on $z\in \H$ by the Mobius transformation
\[ \left(\begin{matrix} a & b \\ c & d \end{matrix}\right) \cdot z = \frac{az+b}{cz+d}.\]
The stabilizer $K$ is $\mathrm{SO}(2,\R)\cong U(1)$ and the algebraic homomorphism is
\[ u: U(1) \mapsto \mathrm{SO}(2,\R), z=a+ib \mapsto \left(\begin{matrix} a & b \\ -b & a \end{matrix}\right).\]
The group $\M_\R$ is $\mathbb{G}_m\times \mathrm{SL}(2,\R)/\langle -1\times -\Id \rangle \cong \{ A\in\mathrm{GL}(2,\R) \colon \det A>0\}$.
The algebraic homomorphism associated to $i$ is
\[ h: \mathbb{S} \to \M, z=a+ib \mapsto \left[\left(\begin{matrix} a & b \\ -b & a \end{matrix}\right)\right]. \] 
Hence one can identify $\H$ with the $\mathrm{GL}^{+}(2,\R)$-conjugacy class of $h$ and $i$ is identified with $h$. Since $\mathbb{G}_m$ fixes $h$, $\H$ is actually the $\mathrm{SL}^{+}(2,\R)$-conjugacy class.

\textbf{Step 2}. The next step concerns about Harish-Chandra embedding. The Lie algebra of $\mathrm{SL}(2,\R)$ is $\mathfrak{sl}(2,\R)$, whose complexification is $\g_\C=\mathfrak{sl}(2,\C)$. Now for $X=\left(\begin{matrix} x & y \\ z & -x \end{matrix}\right)$, and we have
\begin{align*}
(\mathrm{Ad}\circ h)(re^{i\theta})(X)&=r\left(\begin{matrix} \cos \theta & \sin \theta \\ -\sin\theta & \cos\theta \end{matrix}\right) \cdot X \cdot r^{-1}\left(\begin{matrix} \cos \theta & -\sin \theta \\ \sin\theta & \cos\theta \end{matrix}\right)\\
&=\left(\begin{matrix} x\cos 2\theta+\frac{y+z}{2}\sin2\theta & y\cos^2\theta-z\sin^2\theta-x\sin2\theta \\ z\cos^2\theta-y\sin^2\theta-x\sin2\theta & -x-\frac{y+z}{2}\sin2\theta \end{matrix}\right).
\end{align*}
Hence the Hodge decomposition of $\g_\C$ associated to $\mathrm{Ad}\circ h$ is
\[ \g^{0,0}_{\C}=\C \left(\begin{matrix} 0 & 1 \\ -1 & 0 \end{matrix}\right), \g^{-1,1}_{\C} =\C \left(\begin{matrix} 1 & i \\ i & -1 \end{matrix}\right), \g^{1,-1}_{\C}=\C \left(\begin{matrix} 1 & -i \\ -i & -1 \end{matrix}\right) .
\]
Then the Harish-Chandra embedding is given by
\[ i_{\mathrm{HC}}: \H \hookrightarrow \g^{-1,1}_{\C}, \quad  \tau \mapsto \frac{\tau-i}{\tau+i}\left(\begin{matrix} 1/2 & i/2 \\ i/2 & -1/2 \end{matrix}\right). \]
\begin{remark}
It is easy to see this in the disk model of the upper half plane (see \cite[(3.36)]{Viviani}) and then transform it back.
\end{remark}

\begin{remark}\label{rewrite Harish-Chandra}
Identifying $\g^{-1,1}_\C$ with $\C$ via $\left(\begin{matrix} 1/2 & i/2 \\ i/2 & -1/2 \end{matrix}\right) \mapsto 1$, the Harish-Chandra embedding becomes
\[ \H \to \C, \tau \mapsto \frac{\tau -  i }{(\tau +i) }, \]
whose image is $\Delta$. Now $\mathrm{SL}(2,\C)$ acts via the Mobius transformation on $\mathbf{CP}^1$. With respect to this action, we claim that $\mathrm{Exp}(i_{\mathrm{HC}}(\H))\cdot i=\H$. 

Let $\tau \in \H$ be any point. Set $t=\frac{\tau -  i }{(\tau +i)}$. Then
\begin{align*}
e^{i_{\mathrm{HC}}(\tau)}\cdot i &= \left[\Id +  \frac{\tau -  i }{\tau +i }\left(\begin{matrix} 1/2 & i/2 \\  i/2 & -1/2 \end{matrix}\right)\right] \cdot i \\
&=\left(\begin{matrix} 1+\frac{t}{2} & it/2 \\  it/2 & 1-\frac{t}{2}\end{matrix}\right)\cdot i =\tau.
\end{align*}
\end{remark}

\textbf{Step 3}. Now let $V=\R^2$ and we consider the standard representation of $\mathrm{SL}(2,\R)$ on $V$ by left multiplications. Since $-\Id$ acts on $V$ as $-1$, one can define the representation
\[ \rho: \M_\R=\mathbb{G}_m\times \mathrm{SL}(2,\R)/\langle -1\times -\Id \rangle \to \mathrm{GL}(V)\] 
where $\lambda\Id$ acts as $\lambda^{-1}$. In particular, $V$ is a real Hodge structure of pure weight $1$ with respect to $\rho \circ h$. The Hodge decomposition of $V_h$ is
\begin{align*}
V^{1,0}_h &=\{ v\in V \colon \rho(h(z))(v)=z^{-1}v \}=\C\left(\begin{matrix} i \\ 1 \end{matrix}\right), \\
V^{0,1}_h &=\{ v\in V \colon \rho(h(z))(v)=zv \}=\C\left(\begin{matrix} -i \\ 1 \end{matrix}\right).
\end{align*}
Varying $h'$ in the $\mathrm{GL}^{+}(2,\R)$-conjugacy class of $h$ give rise to $\V^{1,0}$ and $\V^{0,1}$, which are Hodge bundles of $\V=V\otimes \mathcal{O}_\H$. Here the smallest nonzero piece in the Hodge filtration of $\V$ is $E=\V^{1,0}$. $\V$ is polarized by the bilinear form $S: V \times V \to \R$, which corresponds to the matrix $\left(\begin{matrix} 0 & -1 \\ 1 & 0 \end{matrix}\right)$. 

Let $X_i=\C/(\mathbb{Z}+i\mathbb{Z})$ be the standard elliptic curve. Then $H^1(X_i,\R)$ is a real Hodge structure of weight one, which is generated by $dx$ and $dy$. Then we have the isomorphism of polarized real Hodge structures
\[ T: H^1(X_i,\R) \to V, \quad dx \mapsto \left(\begin{matrix} 0 \\ 1 \end{matrix}\right), \quad dy \mapsto \left(\begin{matrix} 1 \\ 0 \end{matrix}\right). \]
In particular, we have
\[ T_\C: H^1(X_i,\C) \to V_\C, \quad dz \mapsto \left(\begin{matrix} i \\ 1 \end{matrix}\right), \quad d\overline{z} \mapsto \left(\begin{matrix} -i \\ 1 \end{matrix}\right) \]
where $H^1(X_i,\C)$ is generated by $dz=dx+idy$ and $d\overline{z}=dx-idy$. 

Let $\pi: X \to \H$ be the universal family of elliptic curves. We have $R^2\pi_{\ast}(\C)$ is isomorphic to $\V$ as a polarized variation of real Hodge structures of weight one. Here $\pi_{\ast}(\omega_{X/\H}) \cong \V^{1,0}=E$.

\textbf{Step 4}. Given $v_h \in V^{1,0}_h$, we would like to construct the $L^2$-minimal holomorphic section of $H^0(\H,E)$. Since the dimension of $V^{1,0}_h$ is one, we can assume $v_h=\left(\begin{matrix} i \\  1 \end{matrix}\right)$. Let $\tau \in \H$ be any point. From \textbf{Step 2} we see that $i_{\mathrm{HC}}(\tau)=t\left(\begin{matrix} 1/2 & i/2 \\ i/2 & -1/2 \end{matrix}\right)$ and 
\[ e^{i_{\mathrm{HC}}(\tau)} = \Id +t \left(\begin{matrix} 1/2 & i/2 \\ i/2 & -1/2 \end{matrix}\right)=\left(\begin{matrix} 1+\frac{t}{2} & it/2 \\  it/2 & 1-\frac{t}{2} \end{matrix}\right),\]
where $t =\frac{\tau-i}{\tau+i}$,
Then we define $\sigma \in H^0(\H, E)$ as in Proposition \ref{prop:equivariant sections} by
\[ \sigma(\tau) \colonequals \left(\begin{matrix} 1+\frac{t}{2} & it/2 \\  it/2 & 1-\frac{t}{2} \end{matrix}\right) \cdot \left(\begin{matrix} i \\ 1 \end{matrix}\right)=\left(\begin{matrix} i(1+t) \\  1-t \end{matrix}\right)=\frac{2i}{\tau+i}\left(\begin{matrix} \tau \\  1 \end{matrix}\right). \]
Clearly, this is a holomorphic section of $\V$. Geometrically, 
\[ \sigma(\tau)=\frac{2i}{\tau+i}(dx+\tau dy)=\frac{2i}{\tau+i}(dz)_{\tau},\]
as we mentioned in the introduction. One also see that
\[ \frac{\norm{\sigma(\tau)}^2}{\norm{v_h}^2} = \left|\frac{2i}{\tau+i}\right|^2 \cdot \mathrm{Im}\tau=\abs{1-t}^2\cdot \mathrm{Re}\left(\frac{1+t}{1-t}\right)=1-\abs{t}^2.\]
Therefore
\[ \frac{\norm{\sigma}^2}{\norm{v_h}^2}=\int_\H (1-\abs{t}^2)\phi^{\ast}(d\mu) =\int_{\Delta} (1-\abs{t}^2) d\mu=\frac{\pi}{2}.\]
Here $\phi:\H \to \Delta, \tau \mapsto \frac{\tau-i}{\tau+i}$ is the Harish-Chandra embedding by Remark \ref{rewrite Harish-Chandra} and $d\mu$ is the standard Lebesgue measure on $\Delta$.

\textbf{Step 5}. Lastly, we want to verify that $\sigma$ is indeed a holomorphic section of $E=\V^{1,0}$. Let $\tau=a+ib \in \H$. Consider the matrix
\[ A=\left(\begin{matrix} \sqrt{b} & a/\sqrt{b} \\ 0 & 1/\sqrt{b} \end{matrix}\right) \in \mathrm{SL}^{+}(2,\R),\]
and
\[ A\cdot i =a+ib=\tau.\]
In particular, let $h'$ be a point in the $\mathrm{SL}^{+}(2,\R)$-conjugacy class of $h$ corresponding to $\tau$, then
\[ V^{1,0}_{h'} = \C A\left(\begin{matrix} i \\ 1 \end{matrix}\right) =\C \left(\begin{matrix} \tau \\ 1 \end{matrix}\right).\]
Hence we see that $\sigma(h')=\sigma(\tau)=\frac{2i}{\tau+i}\left(\begin{matrix} \tau \\  1 \end{matrix}\right)\in V^{1,0}_{h'}$.

\bibliographystyle{amsalpha}
\bibliography{SOT}{}
\end{document}